\documentclass{amsart}

\usepackage{amssymb}
\usepackage{graphicx}
\usepackage{enumerate}
\usepackage{amscd}
\usepackage{color}
\usepackage[all]{xy}

\title{$C^\infty$-logarithmic transformations and generalized complex structures
}
\author{Ryushi Goto}
\address{Department of Mathematics, Graduate School of Science, Osaka University, Toyonaka, Osaka 560-0043, Japan}
\email{goto@math.sci.osaka-u.ac.jp}
\author{Kenta Hayano}
\address{Department of Mathematics, Graduate School of Science, Osaka University, Toyonaka, Osaka 560-0043, Japan}
\email{k-hayano@cr.math.sci.osaka-u.ac.jp}

\theoremstyle{plain}
\newtheorem{theorem}{Theorem}[section]
\newtheorem{corollary}[theorem]{Corollary}
\newtheorem{lemma}[theorem]{Lemma}
\newtheorem{proposition}[theorem]{Proposition}

\theoremstyle{definition}
\newtheorem{definition}[theorem]{Definition}

\newtheorem{remark}[theorem]{Remark}



\def\Int{\operatorname{Int}}
\def\Diff{\operatorname{Diff}}

\def\id{\operatorname{id}}

\def\J{\mathcal{J}}
\def\C{\mathbb{C}}
\def\ol{\overline}
\def\ome{\omega}
\def\w{\wedge}

\begin{document}

\begin{abstract}
We provide a new construction of generalized complex manifolds by every logarithmic transformations.
Applying a technique of broken Lefschetz fibrations, we obtain twisted generalized complex structures with 
arbitrary large numbers of connected components of type changing loci on the manifold which is obtained from a symplectic manifold
by logarithmic transformations of multiplicity $0$ on a symplectic $2$-torus with trivial normal bundle.
Elliptic surfaces with non-zero euler characteristic and 
the connected sums $(2m-1)S^2\times S^2$, $(2m-1)\C P^2\# l \overline{\C P^2}$ and $S^1\times S^3$ admit
twisted generalized complex structures $\J_n$ with $n$ type changing loci for arbitrary large $n$.
\end{abstract}

\maketitle

\section{Introduction}
Generalized complex structures are geometric structures on a manifold introduced by Hitchin which are hybrids of 
ordinary complex structures and symplectic structures \cite{Hitchin_2003}. 
On a $4$-manifold a generalized complex structure of type number $2$ is induced from 
a complex structure and the one of type number $0$ is obtained from a symplectic structure modulo the action of $B$-fields. 
Type changing loci of a generalized complex $4$-manifold $(X, \J)$ is the set of points on which the type number of $\J$  changes  from $0$ to $2$.
We denote by $\kappa(\J)$ the number of connected components of type changing loci of $\J$

Cavalcanti and Gualtieri \cite{Cavalcanti_Gualtieri_2006, Cabalcanti_Gualtieri_2009} constructed intriguing generalized complex structures $\J$ with $\kappa(\J)=1$ on the connected sums $(2m-1)\C P^2\# l \ol{\C P^2} $ for $m\geq 1, l\geq 0$, which do not admit any complex and symplectic structures if $m\neq 1$.
Torres \cite{Torres_2012} constructed many examples of generalized complex structures with multiple type changing loci. 
They used logarithmic transformations of multiplicity $0$ to obtain 
generalized complex $4$-manifolds.

In this note, 
we generalize their construction to arbitrary logarithmic transformations at first.
A $2$-torus $T$ on a generalized complex $4$-manifold $(X, \J)$ is called {\it symplectic} if 
$\J$ is induced from a symplectic structure $\ome$ \footnote{
A symplectic structure $\omega$ induces the generalized complex structure $\mathcal{J}_\omega = \begin{pmatrix}
0 & -\omega^{-1} \\
\omega & 0
\end{pmatrix}$ . } on 
a neighborhood of $T$ on which $T$ is symplectic with respect to $\ome$.
We start with a generalized complex manifold $(X, \J)$ with $\kappa(\J)=m$ which has symplectic $2$-tori of self-intersection number zero. 
Then it is shown that a nontrivial logarithmic transformation on the symplectic $2$-tori of $X$ yields a new generalized complex structure $\J'$ with 
$\kappa(\J')=m+1$.
We apply suitable logarithmic transformations of multiplicity $1$ successively
which do not change the diffeomorphism type of a manifold.
Thus we obtain 
a family of generalized complex structures $\{\J_n\}$ with $\kappa(\J_n)=n$
for every $n >m$ on a manifold.

If a generalized complex structure gives a generalized K\"ahler $4$-manifold, it turns out  that the number of connected components of type changing loci are less than or equal to two.  In fact, generalized K\"ahler structures are equivalent to bihermitian structures $(J^+, J^-, h, b)$ and the type changing loci are given by the set
$\{ p\in X\, |\, J^+_p=\pm J^-_p\,\}$ which coincides with zeros of a holomorphic Poisson structures \cite{A.G.G_2006, Hitchin_2007, Goto_2010}. 
Thus our result is a sharp contrast to the one of generalized complex structures in generalized K\"ahler manifolds. 

In Section $2$, we give a brief explanation of generalized complex structures. 
In Section $3$, we give our construction of generalized complex manifolds (Theorem \ref{thm_logtransform_GCS}).
Moishezon proves that logarithmic transformations of multiplicity one do not change the diffeomorphism type of a genus $1$-Lefschetz fibration with a cusp neighborhood. 
Then Theorem \ref{thm_logtransform_GCS} yields generalized complex structures $\J_n$ with $\kappa(\J_n)=n$ for 
every $n\geq 0$ on a genus $1$-Lefschetz fibration with a cusp neighborhood. In Section \ref{subsec_elliptic}, we also show that 
the connected sum $(2m-1)\C P^2\# l\overline{\C P^2}$ for $m\geq 1, l\geq 0$ admits generalized complex structures $\J_n$ with $\kappa(\J_n)=n$ for 
every $n\geq 0$.
In Section \ref{subsec_multiplicity0LT}, we develop a technique of broken Lefschetz fibrations\footnote{ Broken Lefschetz fibrations are introduced by Auroux, Donaldson and Katzarkov from the context of degenerate symplectic structures which are studied from view points of topology of $4$-manifolds \cite{A.D.K_2005}} and
show that a certain logarithmic transformation of multiplicity one on a $2$-torus also preserves the diffeomorphism type of a manifold.
This result has an independent importance from view points of topology of $4$-manifolds (see Lemma \ref{lem_logtrans_fold}).   
We apply Lemma \ref{lem_logtrans_fold} to a generalized complex $4$-manifold $X$ with $m$ type changing loci and a symplectic $2$-torus $T$ of self-intersection number zero.
Then it turns out that
 the manifold which is given by logarithmic transformations of multiplicity $0$ on $X$ along $T$
 admits a (twisted) generalized complex structure $\J_n$ with $n$ type changing loci for every $n>m$ (Theorem \ref{thm_GCS_multiplicity0}). 
 We apply Theorem \ref{thm_GCS_multiplicity0} to the connected sum $(2m+1) S^2\times S^2$ and $S^1\times S^3$
(Corollary \ref{cor_GCS_product_sphere} and Corollary \ref{cor_GCS_S1S3}).

After the authors submitted their paper to the Arxiv, Rafael Torres kindly sent us a message that he and Yazinski 
constructed several examples of generalized complex structures with an arbitrary large type changing loci by a different method applying
only logarithmic transformations of multiplicity $0$ \cite{Torres_Yazinski}.

\section{Generalized complex structures}\label{section_GCS}
Let $X$ be a manifold of dimension $2n$ with a closed $3$-form $H$. 
The $H$-twisted {\it Courant bracket} of $TX\oplus T^\ast X$ is given by  
$$
[V + \xi, W + \eta]_H = [V, W] + \mathcal{L}_V\eta - \mathcal{L}_W\xi - \frac{1}{2}d(\eta(V) - \xi (W)) + i_Vi_W H, 
$$
where $[V, W]$ denotes the ordinary bracket of vector fields $V$ and $W$ and 
$\mathcal{L}_V\eta$ is the Lie derivative of a $1$-form $\eta$ by $V$ and 
$\xi(W)$ is the coupling between a vector filed $W$ and a $1$-form $\xi$ and $i_W H$ denotes the interior product.
The bundle $TX \oplus T^\ast X$ inherits the non-degenerate symmetric bilinear form with signature $(2n,2n)$ defined by
$$
\left<V + \xi, W+\eta\right> = \frac{1}{2}(\eta(V) + \xi(W)). 
$$

\begin{definition}
An {\it $H$-twisted generalized complex structure} on $X$ is an endmorphism of $\mathcal{J}$ of the bundle $TX \oplus T^\ast X$ satisfying the following properties: 
\begin{enumerate}
\item $\mathcal{J}^2=-\text{\rm id}$;
\item $\mathcal{J}$ preserves the symmetric bilinear form $\left<\cdot, \cdot\right>$; 
\item the subbundle $L_\mathcal{J}$ of the complexified bundle $(TX\oplus T^\ast X)^\mathbb{C}$ which consists of $\sqrt{-1}$-eigenvectors of $\mathcal{J}$ is closed under the Courant bracket $[\cdot, \cdot]_H$. 
\end{enumerate}
\noindent
We call $\mathcal{J}$ a {\it generalized complex structure} on $X$ if a $3$-form $H$ vanishes at every point in $X$.
\end{definition}
The symmetric bilinear form $\langle\, ,\,\rangle$ yields the Clifford algebra of the complexified bundle $(TX\oplus T^\ast X)^\C$.  
The action of the Clifford algebra of $(TX\oplus T^\ast X)^\C$ on the space of complex differential forms $\wedge^\bullet (T^\ast X)^\C = \oplus_{i=1}^{2n} \wedge^i (T^\ast X)^\mathbb{C}$ is given by
$$
(V +\xi)\cdot \rho = i_V\rho + \xi \wedge \rho,
$$
where $i_V$ denotes the interior product of $\rho$ by a vector filed $V$ and 
$\xi\wedge\rho$ is the wedge product of $\rho$ by a $1$-form $\xi$.
An $H$-twisted generalized complex structure $\mathcal{J}$ yields the {\it canonical line bundle} $K_\mathcal{J}\subset \wedge^\bullet T_\mathbb{C}^\ast X$ which consists of differential forms  
annihilated by elements of $L_\J$.
The canonical line bundle $K_\mathcal{J}$ satisfies the following conditions: 
\begin{enumerate}[(a)]

\item For any point $p\in X$, $\rho_p \in (K_\mathcal{J})_p$ is written as $\exp(B + \sqrt{-1}\omega) \wedge \theta_1 \wedge \cdots \wedge \theta_k$, where $B, \omega $ are real $2$-forms and $\theta_1,\ldots, \theta_k $ are  linearly independent complex $1$-forms.

\item (non-degeneracy condition)  the top degree part $[\rho \wedge \sigma(\overline{\rho})]_{(2n)}\in \wedge^{2n} T_\mathbb{C}^\ast X$ is not equal to $0$ for any $\rho \in K_\mathcal{J}\backslash\{0\}$, where $\overline{\rho}$ denotes the complex conjugate of $\rho$ and
$\sigma: \wedge^\bullet (T^\ast X)^\mathbb{C}\rightarrow \wedge^\bullet (T^\ast X)^\mathbb{C}$ is the involution given by reversal of order of forms, i.e., 
$\theta^1\wedge\theta^2\wedge\cdots\wedge\theta^r\mapsto 
\theta^r\wedge\theta^{r-1}\wedge\cdots\wedge\theta^1$.
\item (integrability condition)  a section $\hat{\rho}$ of $K_\mathcal{J}$ satisfies the following:
$$
d\hat{\rho} + H \wedge \hat{\rho} = (V + \xi)\cdot \hat{\rho},
$$
where  $V + \xi$ is a section of $(TX\oplus T^\ast X)^\mathbb{C}$. 
\end{enumerate}

\begin{definition}
The lowest degree of a differential form in $K_\mathcal{J}\backslash\{0\}$ at a point $p\in X$ is called the {\it type number} of $\mathcal{J}$ at $p$. 
\end{definition}
Conversely, a line bundle $K\subset \wedge^\bullet (T^\ast X)^\mathbb{C}$ satisfying the conditions (a), (b) and (c) gives an $H$-twisted generalized complex structure $\mathcal{J}$ by the following:
 It follows from the condition (a) that the subbundle $L\subset T_\mathbb{C}X\oplus T^\ast_\mathbb{C}X$ which consists of annihilators of $K$ is maximal isotropic with respect to $\left<\cdot, \cdot\right>$. 
The condition (b) is equivalent to the condition $L\cap \overline{L} = \{0\}$ which gives rise to a direct sum decomposition $(TX \oplus T^\ast X)^\mathbb{C} = L \oplus \overline{L}$. 
We define an endmorphism $\mathcal{J}$ on $T_\mathbb{C}X \oplus T^\ast_\mathbb{C}X$ by
\begin{equation*}
\mathcal{J}(\varphi) = \begin{cases}
\sqrt{-1}\varphi & \text{(if $\varphi\in L$)}, \\
-\sqrt{-1}\varphi & \text{(if $\varphi\in \overline{L}$)}.
\end{cases}
\end{equation*}
It follows that $\J$ is compatible with the bilinear form $\left<\cdot, \cdot\right>$. 
The condition (c) implies that $L$ is closed under the Courant bracket $[\cdot, \cdot]_H$. 
Thus $\mathcal{J}$ is an $H$-twisted generalized complex structure on $X$. 

\begin{remark}
If a $4$-manifold $X$ admits an $H$-twisted generalized complex structure $\mathcal{J}$, then $X$ admits an $H^\prime$-twisted generalized complex structure $\mathcal{J}^\prime$ for any closed $3$-form $H^\prime$ which is cohomologous to $H$. 
Indeed, there exists a $2$-form $\theta$ such that $H^\prime$ is equal to $H +d\theta$ and then $K^\prime := \exp(-\theta) \cdot K_\mathcal{J} $ is a line bundle which satisfies the conditions (a) and (b). 
If $\hat{\rho}$ satisfies $d\hat{\rho}+ H\w\hat{\rho}=(V+\xi)\cdot\hat{\rho}$, then it follows that
$$
d(\exp(-\theta)\cdot \hat{\rho}) + H^\prime\wedge (\exp(-\theta)\cdot\hat{\rho}) = 
(V'+\xi')\cdot  (\exp(-\theta)\cdot\hat{\rho}), 
$$where
$(V'+\xi')=\left(\exp(-\theta)\cdot (V+ \xi) \cdot \exp(\theta)\right) \in (TX\oplus T^*X)^{\mathbb{C}}.$
Thus the line bundle $K'$ satisfies the condition (c) for $H^\prime$. 
In particular, if the cohomology group $H^3(X;\mathbb{R})$ is trivial, then any twisted generalized complex structure can be changed into the untwisted one. 
\end{remark}

Let $(X, \mathcal{J})$ be a generalized complex $4$-manifold with even type numbers. 
The natural projection $\wedge^\bullet (T^\ast X)^\mathbb{C} \rightarrow \wedge^0 (T^\ast X)^\mathbb{C} \cong \mathbb{C}$ gives rise to the canonical section $s \in \Gamma(K_\mathcal{J}^\ast)$. 
The type number of $\mathcal{J}$ jumps from $0$ to $2$ when one goes into zeros of the section $s$. 
We call the set of zeros $s^{-1}(0)$ the {\it type changing loci} of $\mathcal{J}$. 
A point $p \in s^{-1}(0)$ is said to be {\it nondegenerate} if $p$ is a nondegenerate zero of $s$. 
Cavalcanti and Gualtieri proved in \cite{Cavalcanti_Gualtieri_2006} that any compact component of the type changing loci of $\mathcal{J}$ which consists of nondegenerate points is a smooth elliptic curve whose complex structure is induced by $\mathcal{J}$.

\section{$C^\infty$-logarithmic transformations for generalized complex $4$-manifolds}
Let $T$ be an embedded torus in a $4$-manifold $X$ with trivial normal bundle. 
We denote by $\nu T$ a regular neighborhood of $T$, which is diffeomorphic to $D^2\times T^2$.
We remove $\nu T$ from $X$ and glue $D^2\times T^2$ by a diffeomorphism $\psi: \partial D^2\times T^2 \rightarrow \partial \nu T$. 
The procedure as above is called 
a {\it $C^\infty$-logarithmic transformation} on $X$ along $T$ which yields a manifold 
$X_\psi=(X\backslash\Int(\nu T)) \cup_\psi (D^2\times T^2)$.
As explained in \cite{Gompf_Stipsicz}, the circle $\psi(\partial D^2\times \{\ast\})\subset \partial D^2\times T^2$ determines the diffeomorphism type of $X_\psi$. 
A logarithmic transformation is said to be {\it trivial} if the circle $\psi(\partial D^2\times \{\ast\})$ is null-homotopic in $\nu T$. 
Note that a trivial logarithmic transformation on $X$ along $T$ does not change the diffeomorphism type of $X$, that is, $X_\psi$ is diffeomorphic to $X$. 
\begin{theorem}\label{thm_logtransform_GCS}
Let  $\mathcal{J}$ be a twisted generalized complex structure on a $4$-manifold $X$ and an embedded $T$ in $X$  a symplectic torus with trivial normal bundle. 
For any diffeomorphism $\psi:\partial D^2\times T^2 \rightarrow \partial \nu T$, the manifold $X_{\psi}=(X\setminus \Int(\nu T)) \cup_{\psi} D^2\times T^2$ obtained from $X$ by a logarithmic transformation on $T$ admits a twisted generalized complex structure $\mathcal{J}_{\psi}$ which coincides with $\mathcal{J}$ on the complement $X_\psi \setminus D^2\times T^2$. 
Moreover, if the logarithmic transformation is not trivial, $\mathcal{J}_{\psi}$ has a non-empty type changing locus $T_\psi$ in $D^2\times T^2\subset X_\psi$ and $\kappa(\J_\psi)=\kappa(\J)+1$.

\end{theorem}

Since $\nu T$ is diffeomorphic to $D^2\times T^2$,  $\psi$ is obtained by  an orientation preserving self-diffeomorphism of $\partial D^2\times T^2$. 
Any matrix $A\in $SL$(3;\mathbb{Z})$ gives an orientation preserving self-diffeomorphism of $\partial D^2\times T^2$ as follows:
$$
\partial D^2\times T^2 \ni (\theta_1, \theta_2, \theta_3) \mapsto (\theta_1, \theta_2, \theta_3)A\in \partial D^2\times T^2, 
$$
where we identify the manifold $\partial D^2\times T^2$ with $\mathbb{R}^3/\mathbb{Z}^3$ and 
$(\theta_1,\theta_2,\theta_3)$ denotes coordinates of  $\partial D^2\times T^2\cong \mathbb{R}^3/\mathbb{Z}^3$ with
$\theta_1\in \partial D^2$ and 
 $(\theta^2,\theta^3)\in T^2$. 
It is known that any orientation preserving self-diffeomorphism of $\partial D^2\times T^2$ is isotopic to the diffeomorphism represented by a matrix in SL$(3;\mathbb{Z})$.

\begin{proof}[Proof of Theorem \ref{thm_logtransform_GCS}]

If the logarithmic transformation determined by $\psi$ is trivial, the statement is obvious. 
We assume that the logarithmic transformation is not trivial. 

Let $\omega_T$ be a symplectic form of $\nu T$ which induces the generalized complex structure $\mathcal{J}|_{\nu T}$ and $T$ is symplectic with respect to $\ome_T$.
By Weinstein's neighborhood theorem \cite{Weinstein}, we can take a symplectomorphism: 
$$
\Theta: (\nu T, \omega_T) \rightarrow (D^2\times T^2, \sigma_C), 
$$
where $D^2$ denotes the unit disk $\{\, z_1\in \mathbb{C}\, |\, |z_1|\leq1\,\}$ and $T^2$ is the quotient $\mathbb{C}/\mathbb{Z}^2 \cong \mathbb{R}^2/\mathbb{Z}^2$ with a coordinate $z_2$
and  $\sigma_C= \sqrt{-1}C (dz_1\wedge d\bar{z}_1 + dz_2 \wedge d\bar{z}_2)$ for 
the constant $C = \frac{1}{2}\int_{T}\omega_T$.
Using this identification, the attaching map $\psi$ can be regarded as a matrix $A_\psi \in $SL$(3;\mathbb{Z})$. 
Any matrix $P\in $SL$(2;\mathbb{Z})$ induces a self-diffeomorphism $P: T^2\rightarrow T^2$. 
Since the map $\id_{D^2}\times P$ preserves the form $\sigma_C$ and the diffeomorphism type of $X_\psi$ is determined by the first row of $A_\psi$, we can assume that $A_\psi$ is equal to the following matrix: 
$$
\begin{pmatrix}
m & 0 & p \\
0 & 1 & 0 \\
a & 0 & b
\end{pmatrix}
$$
(see \cite{Gompf_Stipsicz}),
where $m,p,a,b$ are integers which satisfy $mb-pa=1$. 
The first row of $A_\psi$ is $(m,0,p)$ and it follows that $p\neq 0$ since $\psi$ is not trivial.
Thus we can take $a$ and $b$ satisfying the condition 
\begin{equation}\label{eq:condition} mbs -pa \neq 0\quad \text{\rm for all } s\in[0,1] \end{equation} by replacing 
$(a,b)$ by $(a+ml, b+pl)$ for a suitable integer $l$ if necessary.
Let $D_k $ be the annulus $\{z_1 \in \mathbb{C} \hspace{.3em}|\hspace{.3em} k<|z|\leq 1\}$ for $k \in [0,1]$. 
We define a diffeomorphism $\Psi: D_{\frac{1}{e}} \times T^2 \rightarrow D_{0}\times T^2$ as follows: 
$$
\Psi(r, \theta_1,\theta_2, \theta_3) = (\sqrt{\log(er)}, m\theta_1+a\theta_3, \theta_2, p\theta_1+b\theta_3), 
$$
where $z_1=r \exp(\sqrt{-1}\theta_1)$ and $z_2=\theta_2 + \sqrt{-1}\theta_3$. 
The manifold $X_{\psi}$ is diffeomorphic to the following manifold: 
$$
X_{\Psi} = (X\setminus\Int(\nu T)) \cup_{\Psi} D^2\times T^2. 
$$
Thus it suffices to construct a twisted generalized complex structure on $X_\Psi$ which satisfies the conditions in Theorem 3.1. 

We denote by $\varphi_T\in \wedge^{\text{even}}D^2\times T^2$ the following form: 
\begin{equation*}
z_1 \exp\left(-\frac{mC}{2}\varrho(|z_1|^2)\frac{dz_1}{z_1}\wedge \frac{d\bar{z}_1}{\bar{z}_1} - bC dz_2\wedge d\bar{z}_2 + C\left( \frac{a}{2} -p \right) \frac{dz_1}{z_1}\wedge dz_2 -C\left( \frac{a}{2}+p \right) \frac{dz_1}{z_1} \wedge d\bar{z}_2\right), 
\end{equation*}
where $\varrho: \mathbb{R}\rightarrow [0,1]$ is a monotonic increasing function which satisfies $\varrho(r)=0$ if $|r|<\frac{1}{2e^2}$ and $\varrho(r)=1$ if $|r|\geq \frac{1}{e^2}$. 
The form $\varphi_T$ satisfies the condition (a) in Section \ref{section_GCS}. 
It follows form the condition (\ref{eq:condition})  that the top-degree part $[\varphi_T \wedge \sigma(\overline{\varphi_T})]_{(4)}$ is not trivial. 
Since ${z_1}^{-1}\varphi_T$ is $d$-closed on $D_0\times T^2$, the form $\varphi_T$ satisfies the condition (c) in Section \ref{section_GCS}. 
Thus $\varphi_T$ gives a generalized complex structure on $D^2\times T^2$. 
Denote by $B$ and $\omega$ the real part and the imaginary part of the degree-$2$ part of the form ${z_1}^{-1}\varphi_T$, respectively. 
Then it follows from a direct calculation that the pullback $\Psi^\ast \sigma_C$ is equal to $\omega$. 
We take a monotonic decreasing function $\tilde{\varrho}:\mathbb{R}\rightarrow [0,1]$ which satisfies $\tilde{\varrho}(r)=1$ for $|r|<\frac{1}{2}$ and $\tilde{\varrho}(r)=0$ for $|r|\geq 1-\varepsilon$, where $\varepsilon>0$ is a sufficiently small number. 
We define a $2$-form $\tilde{B} \in \Omega^2(X\setminus T)$ by 
\begin{equation*}
\tilde{B} = \begin{cases}
{\Psi^{-1}}^\ast (\tilde{\varrho}(|z_1|^2) B) & \text{on }\nu T\setminus T, \\
0 & \text{on }X\setminus \nu T.  
\end{cases}
\end{equation*}
Then the manifold $X\setminus T$ admits a twisted generalized complex structure $\mathcal{J}^\prime_\psi$ such that $\exp(\tilde{B}+ \sqrt{-1}\omega_T) \in \Omega_\mathbb{C}^\bullet(\nu T\setminus T)$ is a local section of the canonical bundle $K_{\mathcal{J}^\prime_\psi}$.  
Since $\varphi_T$ gives a generalized complex structure on $D^2\times T^2$, the form $\exp((\tilde{\varrho}(|z_1|^2)-1)B) \varphi_T$ induces a $(-d(\tilde{\varrho}(|z_1|^2)B))$-twisted generalized complex structure. 
Since the $2$-form $\Psi^\ast(\tilde{B}+ \sqrt{-1}\sigma_C)$ is equal to $\tilde{\varrho}(|z_1|^2)B+\sqrt{-1}\omega$, we obtain a twisted generalized complex structure on $X_\Psi$ which satisifies the conditions in Theorem 3.1.
This completes the proof of Theorem \ref{thm_logtransform_GCS}. 
\end{proof}

\section{Numbers of components of type changing loci}
Let $(X, \mathcal{J})$ be a generalized complex $4$-manifold and $T\subset X$ a symplectic torus with respect to $\mathcal{J}$. 
According to Theorem \ref{thm_logtransform_GCS}, the $4$-manifold $X^\prime$ obtained from $X$ by a logarithmic transformation on $T$ admits a twisted generalized complex structure. 
Moreover, if the logarithmic transformation is not trivial, the number of components of type changing loci is increased by $1$. 
Using this observation, we will prove in this section that several $4$-manifolds admit twisted generalized complex structures with arbitrarily many components of type changing loci.

\subsection{Elliptic surfaces}\label{subsec_elliptic}

Let $f: FT\rightarrow D^2$ be a genus-$1$ Lefschetz fibration over $D^2$ with a Lefschetz singularity whose vanishing cycle $c$ is non-separating. 
The total space $FT$ of $f$ is called a {\it fishtail neighborhood}. 
The manifold $FT$ contains a torus $T\subset FT$ as a regular fiber of the fibration $f$. 
We denote by $m\subset \partial\nu T$ a circle which bounds a section of the fibration $f|_{\nu T}: \nu T \rightarrow D^2$. 
For a non-zero integer $k \in \mathbb{Z}$, we take an orientation preserving diffeomorphism $\varphi_k$ of $\partial \nu T$ which maps the homology class $[m]\in H_1(\nu T;\mathbb{Z})$ to $[m] + k[c]$. 
Let $FT^\prime_k$ be a manifold obtained by removing $\nu T$ from $FT$ and gluing it back using $\varphi_k$. 
Gompf \cite{Gompf_2010} proved that there exists an orientation and fiber preserving diffeomorphism from $FT$ to $FT^\prime_k$ whose restriction on the boundary is the identity map for every $k$. 
Every elliptic surface with non-zero Euler characteristic contains a fishtail neighborhood as a submanifold (see \cite[Theorem 8.3.12]{Gompf_Stipsicz}, for example). 
Furthermore, a regular fiber $T\subset FT$ in an elliptic surface is symplectic with respect to the symplectic structure of the elliptic surface. 
We eventually obtain: 

\begin{theorem}\label{thm_GCS_ellipticsurface}

For any non-negative integer $n$, every elliptic surface with non-zero Euler characteristic admits a generalized complex structure with $n$ components of type changing loci. 

\end{theorem}
\begin{proposition}\label{prop_GCS_connectedsums}
For every $k, l \geq 0$ and $m \geq 1$, the connected sum $(2k+1)\C P^2\# l\ol{\C P^2}$ admits generalized complex structures $\J_m$ 
with $\kappa(\J_m)=m$.
\end{proposition}

\begin{proof}

Cavalcanti and Gualtieri constructed a generalized complex structure on the connected sum $(2n-1)\C P^2\#(10n-1)\ol{\C P^2}$ which is obtained from 
the Lefschetz fibration $E(n)$ by a multiplicity $0$ logarithmic transformation on a regular fiber of the fibration. 
They further constructed a generalized complex structure $\J_1$ on the manifold $(2n-1)\C P^2$ by blowing down $(2n-1)\C P^2\#(10n-1)\ol{\C P^2}$ (see \cite[Example 5.3]{Cabalcanti_Gualtieri_2009} for details). 
The Lefschetz fibration $E(n)$ contains a fishtail neighborhood. 
It is easy to verify that we can make the support of all the above surgeries away from the fishtail neighborhood. 
Thus, we can apply multiplicity $1$ logarithmic transformations on regular fibers of the fishtail neighborhood arbitrarily many times so that the diffeomorphism type of $(2n-1)\C P^2$ is unchanged. 
In particular, we can construct generalized complex structures $\J_m$ of $(2n-1)\C P^2$ with $\kappa(\J_m) = m$. 
Using $\J_m$, we can construct a generalized complex structure of the manifold $(2n-1)\C P^2 \# l \ol{\C P^2}$ with $m$ components of type changing loci for arbitrarily large $m$ (see \cite[Theorem 3.3]{Cabalcanti_Gualtieri_2009}). 
This completes the proof of Proposition \ref{prop_GCS_connectedsums}. 
\end{proof}

\subsection{Multiplicity $0$ logarithmic transformations}\label{subsec_multiplicity0LT}

We take a small ball $B^3 \subset \Int(S^1\times D^2)$ and denote by $M$ the manifold $S^1\times D^2 \setminus \Int(B^3)$. 
Let $h: M\rightarrow I=[0,1]$ be a Morse function satisfying the following properties: 
\begin{enumerate}[(a)]

\item $h^{-1}(0) = \partial B^3$;

\item $h^{-1}(1) = \partial(S^1\times D^2)$;

\item $h$ has a unique critical point $p_0$ in $h^{-1} (\frac{1}{2})$ whose index is equal to $1$. 

\end{enumerate}
We denote by $Y$ the $4$-manifold $M\times I$. 
Define a map $f: Y\rightarrow I\times I$ as follows: 
$$
f(x,t) = (h(x), t) \hspace{.3em} ((x,t) \in M\times I). 
$$
The set of critical points of $f$ is $\{p_0\}\times I$ and that every critical point of $f$ is an indefinite fold. 
A regular fiber of $f$ in the higher-genus side of the folds is a torus. 

\begin{lemma}\label{lem_submfd_multi0logtrans}

Let $X$ be a $4$-manifold which contains an embedded torus $T\subset X$ with self-intersection $0$. 
Then, the manifold $X^\prime$ obtained from $X$ by a multiplicity $0$ logarithmic transformation on $T$ contains $Y$ as a submanifold. 
Furthermore, if $X$ admits a twisted generalized complex structure which makes $T$ symplectic, then a regular fiber $T^2\subset \Int(Y)$ of $f$ is also symplectic with respect to the induced twisted generalized complex structure on $X^\prime$. 

\end{lemma}

\begin{proof}

We take an identification $\nu T \cong D^2\times T^2$. 
Take a small regular neighborhood $\nu T^\prime \subset \nu T$ of $\{0\}\times T^2$. 
Denote by $\widehat{\nu T}$ the manifold obtained from $\nu T$ by a multiplicity $0$ logarithmic transformation on $\{0\}\times T^2$, that is, $\widehat{\nu T} = (\nu T \setminus \Int(\nu T^\prime)) \cup_{\varphi} \nu T^\prime$. 
We can take a handle decomposition of $\widehat{\nu T}$ so that $\nu T^\prime$ is decomposed into a $2$-handle $h_2$, two $3$-handles $h_3^1, h_3^2$ and a $4$-handle $h_4$. 
Using the method in \cite[Subsection 6.2]{Gompf_Stipsicz}, we can draw a Kirby diagram of $(\nu T \setminus \Int(\nu T^\prime)) \cup h_2$ as described in the left side of Figure~\ref{fig_multiplicity0_logtrans}. 
According to \cite[Section 2]{Baykur_2009}, a Kirby diagram of $Y$ is as in the right side of Figure~\ref{fig_multiplicity0_logtrans}. 
\begin{figure}[htbp]
\begin{center}
\includegraphics[width=80mm]{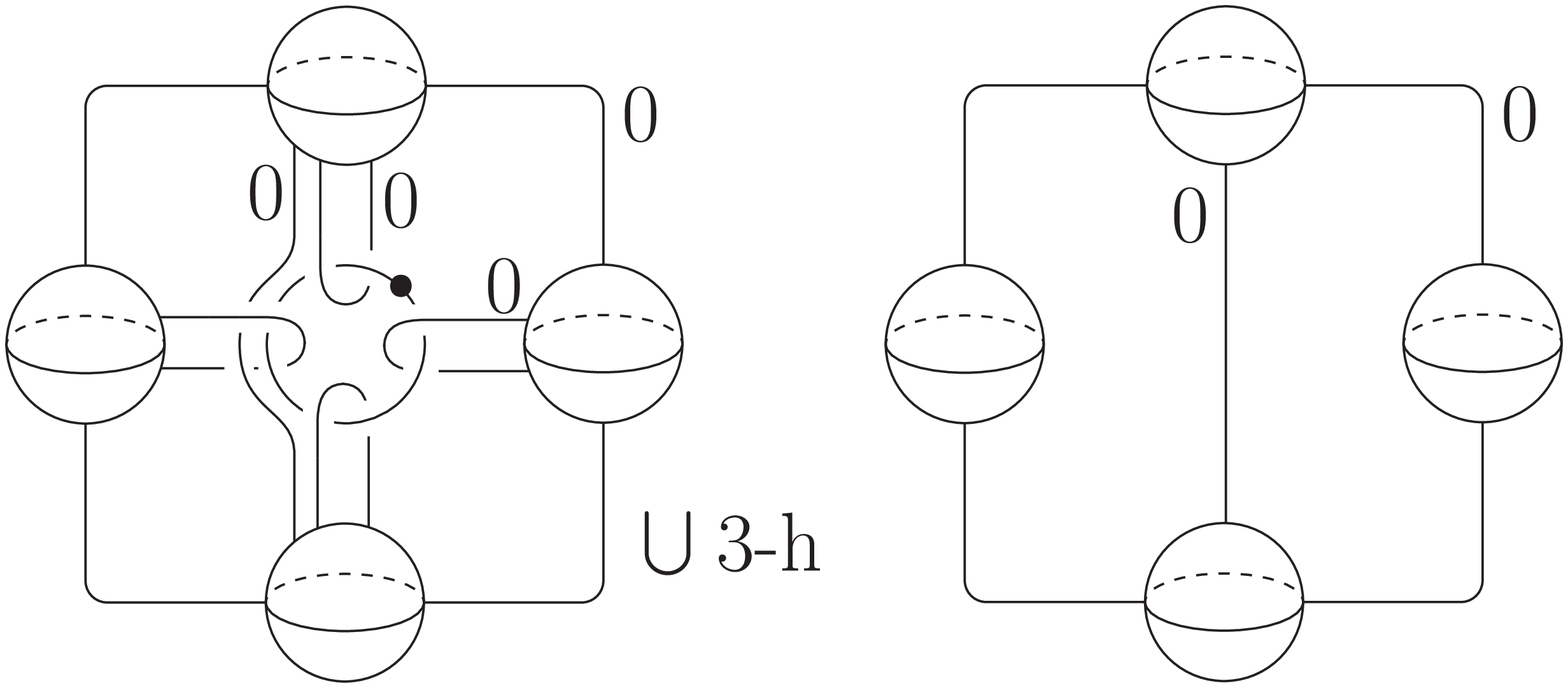}
\end{center}
\caption{The left picture is a Kirby diagram of $(\nu T \setminus \Int(\nu T^\prime)) \cup h_2$, while the right one is a Kirby diagram of $Y$. }
\label{fig_multiplicity0_logtrans}
\end{figure}
By these diagrams, it is clear that the manifold $(\nu T \setminus \Int(\nu T^\prime)) \cup h_2$ contains $Y$ as a submanifold. 
Furthermore, these diagrams also show that the torus $T^2\subset Y \subset X^\prime$ is contained in $\nu T\setminus \Int(\nu T^\prime)$ as a symplectic torus. 
This completes the proof of Lemma \ref{lem_submfd_multi0logtrans}. 
\end{proof}

Let $\alpha \subset T^2$ be a vanishing cycle of the folds of $f$. 
We denote by $m\subset \partial(\nu T^2)$ a circle which bounds a section of a fibration $f|_{\nu T^2}: \nu T^2 \rightarrow f(\nu T^2)$. 

\begin{lemma}\label{lem_logtrans_fold}

Let $\varphi$ be an orientation preserving self-diffeomorphism of $\partial \nu T^2$ such that the induced map $\varphi_\ast$ on the first homology group maps the homology class $[m]\in H_1(\partial \nu T^2;\mathbb{Z})$ to the class $[m]+2n[\alpha]$, where $n$ is an integer. 
We denote by $Y^\prime$ the manifold $(Y\setminus \Int(\nu T^2)) \cup_{\varphi} \nu T^2$. 
Then, there exists an orientation preserving diffeomorphism $\Phi: Y\rightarrow Y^\prime$ satisfying the following properties: 
\begin{enumerate}

\item the restriction $\Phi|_{\partial Y}$ is the identity map (note that this condition makes sense since $\partial Y$ is equal to $\partial Y^\prime$ as a set); 

\item the following diagram commutes: 
\begin{equation*}
	\xymatrix{
	Y \ar[r]^{\Phi}\ar[d]_{f} & Y^\prime \ar[dl]^{f^\prime} \\
	I\times I,
	} 
	\end{equation*}
where $f^\prime:Y^\prime \rightarrow I\times I$ is the fibration induced by $f$. 

\end{enumerate}

\end{lemma}

The idea of the proof of this lemma partially relies on that of \cite[Lemma 3.36]{Behrens_preprint}. 
In the proof of \cite[Lemma 3.36]{Behrens_preprint} Behrens extended a self-diffeomorphism of the boundary $\partial \nu T^2$ to the lower-genus side. 
We further have to make this extension trivial on the boundary in the lower-genus side. 

\begin{proof}

For a sufficiently small positive number $\varepsilon >0$, we take an open neighborhood $U_0\subset M$ of $p_0$ and a diffeomorphism
$$
\varphi_0: U_0 \rightarrow \{ (x,y,z)\in \mathbb{R}^3 \hspace{.3em}|\hspace{.3em} x^2+y^2+z^2 < \varepsilon, |x^2+y^2-z^2|< \varepsilon\}
$$
such that the following diagram commutes: 
$$
\xymatrix{
U_0 \ar[r]^{\varphi_0} \ar[d]_{h} & \mathbb{R}^3 \ar[dl]^{q} \\
(-\varepsilon, \varepsilon),
} 
$$
where $q:\mathbb{R}^3\rightarrow \mathbb{R}$ is a quadratic form defined as $q(x,y,z) = x^2+y^2-z^2$. 
We take a Riemannian metric $g_M$ of $M$ so that it coincides on $U_0$ with the pull-back $\varphi_0^\ast(dx^2+dy^2+dz^2)$. 
For any point $p\in M\setminus \{p_0\}$, we denote by $\mathcal{H}_p \subset T_pM$ the orthogonal complement of the subspace $\ker(dh_p)$. 
Note that $\mathcal{H} = \{\mathcal{H}_p\}_{p\in M\setminus \{p_0\}}$ is a horizontal distribution of the fibration $h_{M\setminus \{p_0\}}$. 
Let $\gamma: I \rightarrow I$ be a path defined as $\gamma(s)=1-s$. 
For a point $p\in h^{-1}(1)$, we denote by $\hat{\gamma}_p: I_p\rightarrow M$ the horizontal lift of $\gamma$ whose initial point is $p$, where $I_p\subset I$ is the maximal subinterval. 
We take a subset $V\subset h^{-1}(1)$ as follows: 
$$
V=\left\{p \in h^{-1}(1)~\left|~\hat{\gamma}_p\left(\frac{1-\varepsilon}{2}\right) \right. \in U_0, \left|p_3\circ \varphi_0\left(\hat{\gamma}_p\left(\frac{1-\varepsilon}{2}\right)\right)\right|< \frac{\varepsilon}{2} \right\}, 
$$
where $p_3: \mathbb{R}^3\rightarrow \mathbb{R}$ is the projection onto the third component. 
We take a diffeomorphism $\varphi_1: h^{-1}(1)\rightarrow \mathbb{R}^2/\mathbb{Z}^2$ which satisfies the following conditions: 
\begin{enumerate}

\item $\varphi_1(V) = (\mathbb{R} \times (-\frac{\varepsilon}{2}, \frac{\varepsilon}{2})) /\mathbb{Z}^2$; 

\item $\varphi_0\circ \Phi\circ \varphi_1^{-1}(s_1,s_2)=(\sqrt{\frac{\varepsilon}{2}-s_2^2}\cos(2\pi s_1), \sqrt{\frac{\varepsilon}{2}-s_2^2}\sin(2\pi s_1), s_2)$, where $\Phi: h^{-1}(1) \rightarrow h^{-1}(\frac{1+\varepsilon}{2})$ is the parallel transport determined by the distribution $\mathcal{H}$. 

\end{enumerate}

We denote by $D\subset M$ the descending manifold of $h$ with respect to $\mathcal{H}$, that is, $D$ is the union of $\{p_0\}$ and the set of points in $M$ which converges to $p_0$ when we take a parallel transport from $h^{-1}(0)$ to $h^{-1}(\frac{1}{2})$ using $\mathcal{H}$. 
We take a submanifold $W\subset M$ as follows: 
$$
W = D \cup \{\hat{\gamma}_p(t)\in M \hspace{.3em}|\hspace{.3em} p \in V,  t\in I_p\}. 
$$
We define a diffeomorphism $\varphi_2: (\mathbb{R}\times[\frac{\varepsilon}{2}, 1-\frac{\varepsilon}{2}])/\mathbb{Z}^2 \times I \rightarrow M \setminus W$ as follows: 
$$
\varphi_2(s_1,s_2,t) = \hat{\gamma}_{\varphi_1^{-1}(s_1,s_2)}(t). 
$$
We can also obtain a diffeomorphism $\varphi_3: \mathbb{R}^2/\mathbb{Z}^2\times [0,\frac{1}{2}) \rightarrow h^{-1}((\frac{1}{2},1])$ in the same way as we obtain $\varphi_2$. 
In the rest of the proof of Lemma \ref{lem_logtrans_fold}, we identify $M\setminus W$ (resp. $h^{-1}((\frac{1}{2},1])$) with $(\mathbb{R}\times[\frac{\varepsilon}{2}, 1-\frac{\varepsilon}{2}])/\mathbb{Z}^2 \times I$ (resp. $\mathbb{R}^2/\mathbb{Z}^2\times [0,\frac{1}{2})$) using $\varphi_2$ (resp. $\varphi_3$). 

Denote by $C$ the rectangle $[\frac{3}{4}-\varepsilon,\frac{3}{4}+\varepsilon]\times[\frac{1}{2}-\varepsilon,\frac{1}{2}+\varepsilon] \subset I\times I$. 
We define a self-diffeomorphism $\psi$ of $f^{-1}(\partial C)$ as follows: 
\begin{itemize}
\item For a point $\left(s_1, s_2, \frac{1}{4}+\varepsilon, u\right)\in f^{-1}\left(\left\{\frac{3}{4}-\varepsilon\right\}\times \left[\frac{1}{2}-\varepsilon,\frac{1}{2}+\varepsilon\right]\right)$, we put 
$$
\psi\left(s_1, s_2, \frac{1}{4}+\varepsilon, u\right) = \left(s_1 + 2n\varrho(u), s_2 , \frac{1}{4}+\varepsilon, u\right),
$$
where $\varrho: [0,1] \rightarrow [0,1]$ is a monotonic increasing smooth function whose value is $0$ around $\frac{1}{2}-\varepsilon$ and $1$ around $\frac{1}{2}+\varepsilon$; 

\item $\psi$ is the identity map on the complement of $f^{-1}\left(\left\{\frac{3}{4}-\varepsilon\right\}\times \left[\frac{1}{2}-\varepsilon,\frac{1}{2}+\varepsilon\right]\right)$. 

\end{itemize}
It is sufficient to prove that $\psi$ can be extended to an orientation and fiber preserving self-diffeomorphism of $f^{-1}(I\times I \setminus \Int(C))$ which is the identity map on $\partial Y$. 

We define an orientation and fiber preserving self-diffeomorphism $\Psi_1$ of $f^{-1}\left(\left[\frac{1-\varepsilon}{2},1\right]\times I\setminus \Int(C)\right)$ as follows: 
\begin{itemize}

\item For a point $(s_1,s_2,t,u)\in f^{-1}\left(\left[\frac{1+\varepsilon}{2},\frac{3}{4}-\varepsilon\right]\times I\setminus \Int(C)\right)\cup M\setminus W\times I$, we put: 
$$
\Psi_1(s_1, s_2,t,u) = (s_1 + 2n\varrho(u), s_2 ,t, u). 
$$

\item For a point $(x,y,z,u)\in U_0\times I$, we put: 
$$
\Psi_1(x,y,z,u) = (x\cos(4\pi n \varrho(u)) -y\sin(4\pi n\varrho(u)),x\sin(4\pi n \varrho(u)) +y\cos(4\pi n \varrho(u)), z, u).
$$

\item $\Psi_1$ is the identity map on $f^{-1}\left(\left[\frac{3}{4}-\varepsilon,1\right]\times I\setminus \Int(C)\right)$. 

\end{itemize}
It is easy to verify that the map $\Psi_1$ is well-defined. 

For any $t<\frac{1}{2}$ and $u\in [0,1]$, a fiber $f^{-1}\left(t,u\right)$ and the manifold $D$ intersect at two points. 
Using the distribution $\mathcal{H}$, we can identify $f^{-1}\left(t,u\right)\setminus D$ with $(\mathbb{R}\times (0,1))/\mathbb{Z}^2$. 
In particular, we can identify a fiber $f^{-1}\left(t,u\right)$ with the $2$-sphere $S^2$. 
With this identification understood, we regard the restriction $\Psi_1|_{f^{-1} \left( \left\{\frac{1-\varepsilon}{2} \right\}\times I \right) }$ as a loop in $\Diff^+(S^2)$. 
This loop represents the trivial element in $\pi_1(\Diff^+(S^2), \id)\cong \mathbb{Z}/2\mathbb{Z}$ (this isomorphism is due to the result of Smale \cite{Smale_1959}). 
Hence, we can extend $\Psi_1$ to an orientation and fiber preserving self-diffeomorphism $\Psi_2$ of $Y$ so that the restriction of $\Psi_2$ on $\partial Y$ is the identity map. 
This completes the proof of Lemma \ref{lem_logtrans_fold}. 
\end{proof}

As a corollary of Lemma \ref{lem_logtrans_fold}, we immediately obtain: 

\begin{corollary}\label{cor_logtrans_fold}

Suppose that a $4$-manifold $X$ contains $Y$ as a submanifold. 
Then, we can apply a non-trivial multiplicity $1$ logarithmic transformation on $T^2\subset Y \subset X$ to $X$ so that the diffeomorphism type of $X$ is unchanged under the transformation. 

\end{corollary}

\begin{remark}\label{rem_Glucktwist}

The proof of Lemma \ref{lem_logtrans_fold} also implies that the {\it Gluck twist} on a lower-genus regular fiber $S^2\subset Y\subset X$ of $f$ can be realized by a multiplicity $1$ logarithmic transformation on $T^2\subset Y\subset X$. 
Indeed, it is easy to see that the manifold $(X\setminus \Int(\nu T^2)) \cup_{\varphi} \nu T^2$ is diffeomorphic to the manifold obtained from $X$ by the Gluck twist on $S^2$ when we take an attaching map $\varphi$ so that the induced map $\varphi_\ast$ maps $[m]$ to $[m]+(2n+1)[\alpha]$. 

\end{remark}

Combining Lemmas \ref{lem_submfd_multi0logtrans} and \ref{lem_logtrans_fold} together with Theorem \ref{thm_logtransform_GCS}, we obtain: 

\begin{theorem}\label{thm_GCS_multiplicity0}

Let $(X, \mathcal{J})$ be a $4$-manifold with a twisted generalized complex structure. 
Suppose that $X$ has a symplectic torus $T$ with trivial normal bundle. 
Denote by $m$ the number of the components of the type change loci of $\mathcal{J}$. 
For any integer $n \geq m+1$, the manifold $X^\prime$ obtained from $X$ by a multiplicity $0$ logarithmic transformation on $T$ and the manifold obtained from $X$ by the Gluck twist on $S\subset Y \subset X^\prime$ admit a twisted generalized complex structure with $n$ type changing loci. 

\end{theorem}

\begin{corollary}\label{cor_GCS_product_sphere}

For a sufficiently large integer $n$ and a non-negative integer $m$, the manifold $\# (2m+1) S^2\times S^2$ admits a generalized complex structure with $n$ components of type changing loci. 

\end{corollary}

\begin{proof}

Let $\Sigma_g$ be a closed oriented surface of genus-$g$. 
K\"{a}hler structures of $\Sigma_2$ and $\Sigma_g$ induce that of the product $\Sigma_2\times \Sigma_g$. 
The manifold $\# (2m+1) S^2\times S^2$ can be obtained by successive application of logarithmic transformations on symplectic tori to $\Sigma_2\times \Sigma_{m+2}$ and one of these transformations has multiplicity $0$ (see \cite[Section 3]{Torres_2012}). 
Thus, Corollary \ref{cor_GCS_product_sphere} holds by Theorem \ref{thm_GCS_multiplicity0}. 
\end{proof}

\begin{corollary}\label{cor_GCS_S1S3}

For any positive integer $n$, the manifold $S^1\times S^3$ admits a twisted generalized complex structure with $n$ components of type changing loci. 

\end{corollary}

\begin{proof}

The manifold $S^1\times S^3$ can be obtained from $S^2\times T^2$ by a multiplicity $0$ logarithmic transformation on $\{\ast\}\times T^2$. 
Indeed, we can draw a Kirby diagram of the manifold obtained by this transformation as shown in Figure~\ref{fig_diagram_S1S3} and it is easy to verify that this manifold is diffeomorphic to $S^1\times S^3$.  
\begin{figure}[htbp]
\begin{center}
\includegraphics[width=50mm]{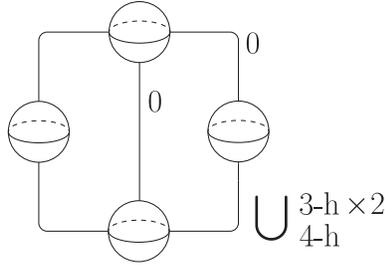}
\end{center}
\caption{A diagram of $S^1\times S^3$.}
\label{fig_diagram_S1S3}
\end{figure}
By Theorem \ref{thm_logtransform_GCS}, $S^1\times S^3$ admits a twisted generalized complex structure $\mathcal{J}_1$ with one component of type changing locus. 
Thus, Corollary \ref{cor_GCS_S1S3} follows from Theorem \ref{thm_GCS_multiplicity0}. 
\end{proof}

\begin{remark}\label{remark_other_GCS}

Torres \cite{Torres_2012} constructed the twisted generalized complex structure $\mathcal{J}_1$ of $S^1\times S^3$ in the proof of Corollary \ref{cor_GCS_S1S3}. 
He further constructed twisted generalized complex structures of $S^2\times T^2$, $S^1\times L(p,1)$, and so on, applying multiplicity $0$ logarithmic transformations on symplectic tori. 
Using Theorem \ref{thm_GCS_multiplicity0}, we can obtain twisted generalized complex structures of these manifolds with $n$ type changing loci for arbitrarily large $n$. 

Torres and Yazinski \cite{Torres_Yazinski} also constructed twisted generalized complex structures of several manifolds such as $S^1\times L(p,1)$ and $T^2\times \Sigma_g$ with arbitrarily large number of connected components of type changing loci. 
In this construction, they used surgery diagrams of $3$-manifolds together with multiplicity $0$ logarithmic transformations. 

\end{remark}

\noindent
{\bf Acknowledgments.}
The authors would like to express their gratitude to Rafael Torres for his helpful comments on generalized complex structures of connected sums of the projective planes and for suggesting that we should apply our results as we explained in Remark \ref{remark_other_GCS}. 
The second author was supported by JSPS Research Fellowships for Young Scientists~(24$\cdot$993).


\begin{thebibliography}{99}
\bibitem[ADK]{A.D.K_2005}
D.~Auroux, S.~ Donaldson and L.~Katzarkov, {\it Singular Lefschetz pencils}, Geom. Topol. 9 (2005), 1043-1114.
\bibitem[AGG]{A.G.G_2006}
V.~Apostolov, P.~Gauduchon and G.~Grantcharov,
{\it Bihermitian structures on complex surfaces},
Pro. London Math. Soc. {\bf 79} (1999), 414-428,
Corrigendum: {\bf 92}(2006), 200-202. 
\bibitem[Ba1]{Baykur_2009} R.~\.{I}.~Baykur, {\it Topology of broken Lefschetz fibrations and near-symplectic 4-manifolds}, Pacific J. Math. \textbf{240}(2009), 201--230.
\bibitem[Be1]{Behrens_preprint} S.~Behrens, {\it On $4$-manifolds, folds and cusps}, to appear in Pacific J. Math.
\bibitem[CG1]{Cavalcanti_Gualtieri_2006} G. R. Cavalcanti and M. Gualtieri, {\it A surgery for generalized complex structures on $4$-manifolds}, J. Differential Geom. \textbf{76}(2006) no. 1, 35--43.
\bibitem[CG2]{Cabalcanti_Gualtieri_2009} G. R. Cavalcanti and M. Gualtieri, {\it Blow-up of generalized complex $4$-manifolds}, J. Topol. \textbf{2}(2009) no. 4, 840--864.
\bibitem[G1]{Gompf_2010} R. E. Gompf, {\it More Cappel-Shaneson spheres are standard}, Algebr. Geom. Topol. {\bf 10}(2010), no. 3, 1665--1681.
\bibitem[GS]{Gompf_Stipsicz} R. E. Gompf and A.I.Stipsicz, {\it 4-Manifolds and Kirby Calculus}, Graduate Studies in Mathematics \textbf{20}, American Mathematical Society, 1999.
\bibitem[Go1]{Goto_2010}
R.~Goto,
{\it Deformations of generalized complex and generalized K\"ahler structures},
J. Differential Geom. 84 (2010), no. 3, 525--560.
\bibitem[Gu1]{Gualtieri_2011} M. Gualtieri, {\it Generalized complex geometry}, Ann. Math. \textbf{2}(2011), no. 1, 75--123.
\bibitem[Hi1]{Hitchin_2003} N.~ Hitchin, {\it Generalized Calabi-Yau manifolds} Q. J. Math. \textbf{54} (2003), no. 3, 281-308.
\bibitem[Hi2]{Hitchin_2007}
N.~Hitchin, 
{\it Bihermitian metrics on del Pezzo surfaces},
J. Symplectic Geom. 5 (2007), no. 1, 1-8. 
\bibitem[S]{Smale_1959} S. Smale, {\it Diffeomorphisms of the $2$-sphere}, Proc. Amer. Math. Soc. \textbf{10}(1959), 621--626.
\bibitem[T]{Torres_2012} R. Torres, {\it Constructions of generalized complex structures in dimension four}, Commun. Math. Phys. \textbf{314}(2012), 351--371.
\bibitem[TY]{Torres_Yazinski} R. Torres and J. Yazinski, {\it On the number of type change loci of a generalized complex structure}, preprint.
\bibitem[W]{Weinstein} A. Weinstein, {\it Symplectic manifolds and their Lagrangian submanifolds}, Advances in Mathematics, \textbf{6}(1971), 329--346.
\end{thebibliography}
\end{document}